\documentclass[a4paper,10pt]{article}
\usepackage[utf8]{inputenc}
\usepackage[T1]{fontenc}
\usepackage[english]{babel}
\usepackage{amsmath}
\usepackage{amsthm}
\usepackage{amsfonts}
\usepackage{amssymb}
\usepackage{listings}
\usepackage{mathrsfs}
\usepackage{graphicx}
\usepackage{hyperref}
\usepackage{algpseudocode}
\usepackage{algorithmicx}
\usepackage{algorithm}
\usepackage{subfig}
\usepackage{stmaryrd}
\usepackage{bm}
\usepackage{tikz}
\usepackage{braket}
\usepackage{wasysym}
\usepackage{wrapfig}
\usepackage[framemethod=TikZ]{mdframed}
\usepackage{xcolor}
\usepackage{pstricks}
\usepackage{faktor}

\usepackage{todonotes}

\newcommand{\mc}{\mathscr}
\newcommand{\f}{\mathbb}

\newcommand{\cu}{\subseteq}
\newcommand{\wh}{\widehat}
\newcommand{\wt}{\widetilde}

\newcommand{\serie}[1]{\{#1_{n}\}_n}
\newcommand{\sdiag}[1]{\{D_{n}(#1)\}_n}
\newcommand{\GLT}{\sim_{GLT}}
\newcommand{\acs}{\xrightarrow{a.c.s.}}
\newcommand{\ve}{\varepsilon}
\newcommand{\B}{\{\{B_{n,m}\}_n\}_m}
\newcommand{\tratti}{\rightharpoonup}
\newcommand{\dacs}[2]{d_{acs}\left(#1,#2\right)}

\DeclareMathOperator{\rk}{rk}

\newtheorem{definition}{Definition}[section]

\newtheorem{lemma}[definition]{Lemma}

\newtheorem{theorem}{Theorem}[section]
\newtheorem*{theorem*}{Theorem}

\title{Diagonal Matrix Sequences and their Spectral Symbols}
\author{Giovanni Barbarino}
\date{}

\begin{document}

\maketitle

\begin{abstract}
The spectral symbols are useful tools to analyse the eigenvalue distribution when dealing with high dimensional linear systems. Given a matrix sequence with an asymptotic symbol, the last one depends only on the spectra of the individual matrices, seen as a not ordered set. We can then focus only on diagonal sequences and sort the eigenvalues so that they become an approximation of the symbol sampling. We show that this is linked to the concept of diagonal Generalized Locally Toeplitz (GLT) sequences, and in particular we prove that any diagonal sequence with a real valued symbol can be permuted in order to obtain a diagonal GLT sequence with the same symbol.

\end{abstract}

\section{Introduction}

A \textit{Matrix Sequence} is an ordered collection of complex valued matrices with increasing size, and is usually denoted as $ \serie A $, where $ A_n\in \f C^{n\times n} $. 

Matrix sequences naturally arise in several contexts. For example, the discretization of
a linear differential or integral equation by a linear numerical method (such as the finite difference method, the finite
element method, the modern isogeometric analysis, etc.) leads to linear systems whose sizes are directly proportional to the accuracy of the method.
The speed of convergence of the solvers (Conjugate Gradient, Preconditioned Krylov methods, Multigrid techniques, etc.) depends on the spectra of the matrices, so the knowledge of the eigenvalue distribution is a strong tool which we can use to
choose or to design the best solver and method of discretization.

It is often observed in practice that the matrix sequences $ \serie A $ generated by discretization methods possess a \textit{spectral symbol}, that is a measurable function describing the asymptotic distribution of the eigenvalues of $A_n$. We recall that a spectral symbol associated with a sequence $\serie A$ is a measurable functions $k:D\cu \f R^n\to \f C$ satisfying 
\[
\lim_{n\to\infty} \frac{1}{n} \sum_{i=1}^{n} F(\lambda_i(A_n)) = \frac{1}{\mu(D)}\int_D F(k(x)) dx
\]
for every continuous function $F:\f C\to \f C$ with compact support, where $D$ is a measurable set with finite Lebesgue measure $\mu(D)>0$. In this case we write 
\[ \serie A\sim_\lambda k(x). \]
We can also consider the singular values of the matrices instead of the eigenvalues. In this case we have that $k:D\cu \f R^n\to \f C$ is a \textit{singular value symbol} of the sequence $\serie A$ if it satisfies
\[
\lim_{n\to\infty} \frac{1}{n} \sum_{i=1}^{n} F(\sigma_i(A_n)) = \frac{1}{\mu(D)}\int_D F(|k(x)|) dx
\]
for every continuous function $F:\f R\to \f C$ with compact support, where $D$ is a measurable set with finite Lebesgue measure $\mu(D)>0$. In this case we write 
\[ \serie A\sim_\sigma k(x). \]
If $\serie A\sim_\lambda k(x)$, then we can consider the diagonal matrices $D_n\in \f C^{n\times n}$ that contain the eigenvalues of $A_n$. We get again that $\serie D\sim_\lambda k(x)$, so we can focus only on diagonal sequences.
Moreover, we will consider only spectral symbol with domain $D=[0,1]$, so from now on 
\[ 
k:[0,1]\to \f C.
 \]

$ $\\

Given a diagonal matrix $D_n \in \f C^{n\times n}$, with diagonal entries $ d^{(n)}_i = [D_n]_{i,i}  $ we can consider the piecewise linear function $k_n:[0,1]\to \f C$ such that interpolate the values $ d^{(n)}_i $ on the nodes $i/n$ and is linear in between.
\[
k_n(0) = 0\qquad k_n\left( \frac{i}{n}\right) = d^{(n)}_i \,\,\forall\,i =1,2,\dots,n
\] 
We say that $D_n$ converges \textit{piecewise} to a function $ k:[0,1]\to \f C $ if the interpolations $k_n(x)$ converge in measure  to $k(x)$. In this case, we write 
\[ \serie D\tratti k(x). \]
 We will give a more rigorous definition later. The first important result regarding this convergence is expressed by the following Theorem that will be proved in Section 3.

\begin{theorem*}
Given a diagonal sequence $\serie D$ and a measurable function $ k:[0,1]\to \f C $, then
\[
 \serie D\tratti k(x)\implies \serie D\sim_\lambda k(x)
\]
\end{theorem*}
The opposite implication does not hold, since the piecewise convergence depends on the order of the eigenvalues, whereas the spectral symbol depends only on their values.

$ $\\

The space of matrix sequences that admit a singular value symbol on a fixed domain $D$ has been shown to be closed with respect to a notion of convergence called the \textit{Approximating Classes of Sequences} (a.c.s.) convergence. This notion and result are due to Serra-Capizzano \cite{ACS}, but were actually inspired by Tilli’s pioneering paper on LT sequences \cite{Tilli}. 
Given a sequence of matrix sequences $\B$, it is said to be a.c.s. convergent to $\serie A$ if there exists a sequence $\{N_{n,m}\}_{n,m}$ of "small norm" matrices and a sequence  $\{R_{n,m}\}_{n,m}$ of "small rank" matrices such that for every $m$ there exists $n_m$ with
\[
A_n = B_{n,m}  + N_{n,m} + R_{n,m}, \qquad \|N_{n,m}\|\le \omega(m), \qquad \rk(R_{n,m})\le nc(m)
\]
for every $n>n_m$, and
\[
\omega(m)\xrightarrow{m\to \infty} 0,\qquad c(m)\xrightarrow{m\to \infty} 0.
\]
In this case, we will use the notation $\B\acs \serie A$. 

It has been observed that the sequences $\serie A$ arising from differential equations can often be obtained through an a.c.s. limit of sums of products of special diagonal and Toeplitz sequences, for which we can easily deduce the spectral symbol. This justifies the interest in the space of Generalized Locally Toeplitz (GLT) sequences, which contains both the before-mentioned class of sequences, gains the structure of a $\f C$-algebra, and is closed with respect to the a.c.s. convergence. We will report only a few properties of this space, but for a detailed presentation of the GLT sequences and their applications refer to \cite{BS},\cite{GLT},\cite{Tilli},\cite{GSM},\cite{GS} and references therein.

For every GLT matrix sequence $\serie A$ one of its singular value symbol $k(x)$ is chosen, called \textit{GLT symbol}, and denoted as $\serie A\GLT k(x)$. Key examples of diagonal GLT sequences are $\sdiag{a}$, where $a:[0,1]\to \f C$ is an almost everywhere (a.e.) continuous function, and
\[
D_n(a) = \underset{i=1,\dots,n}{\text{diag}} a\left(\frac{i}{n}\right) = 
\begin{pmatrix}
a\left(\frac{1}{n}\right) & & & \\
& a\left(\frac{2}{n}\right) & & \\
& & \ddots & \\
& & & a(1)
\end{pmatrix}
\]
It is easy to verify that $\sdiag a \sim_\sigma a(x)\otimes 1$, where 
\[
a(x) \otimes 1 : [0,1]\times [-\pi,\pi] \to \f C
\]
is a two-variable function that is constant with respect to the second variable. This function is chosen as the GLT symbol of $\sdiag a$.  

In the case of diagonal matrix sequences, the choice of one symbol can be seen as a particular sorting of their eigenvalues, as expressed in the following theorems, proved in Section 3 and 4.

\begin{theorem*}
Given a diagonal sequence $\serie D$ and a measurable function $ k:[0,1]\to \f C $, then
\[
 \serie D\tratti k(x)\iff \serie D\GLT k(x)\otimes 1
\]
\end{theorem*}

\begin{theorem*}
Given a real diagonal sequence $\serie D$ and one of its spectral symbols $ k:[0,1]\to \f R $, then
$$\{P_nD_nP_n^T\}\GLT k(x)\otimes 1$$
where $P_n$ are permutation matrices.
\end{theorem*}

The theory of piecewise convergence is not enough in order to prove the complex version of the last result, since the field of complex numbers does not have a natural order. The result holds nonetheless, and we will prove it in a separate document.

\section{A.c.s. and GLT} 

In this section, we introduce some technical results that will be used in the other sections. Moreover, we recall some known lemmas and theorems about spectral symbols, acs convergence and GLT sequences.

\subsection{A.c.s. pseudometric}

Given a matrix $A\in\f C^{n\times n}$, we can define the function
\[
p(A):= \min_{i=1,\dots,n+1}\left\{ \frac{i-1}{n} + \sigma_i(A) \right\}
\]
where $\sigma_1(A)\ge \sigma_2(A)\ge \dots \ge \sigma_n(A)$ are the singular values of $A$, and by convention $\sigma_{n+1}(A)=0$. Another way to compute this quantity makes use of the singular value decomposition (svd) of $A$. In fact, if $A=U\Sigma V$ where $\Sigma$ is a diagonal matrix containing the singular vales of $A$ sorted in a decreasing order, we can decompose it into the sum of two diagonal matrices
\[
\Sigma= 
\begin{pmatrix}
\sigma_1 & & & \\
& \sigma_2 & & \\
& & \ddots & \\
& & & \sigma_n
\end{pmatrix}
= \wt{\Sigma}^{(i)} + \wh{\Sigma}^{(i)}
\]

\[
\wt\Sigma^{(i)}= 
\begin{pmatrix}
\sigma_1 & & & & & \\
& \ddots & & & &\\
& & \sigma_{i-1} & & &\\
& & & 0 & & \\
& & &  & \ddots & \\
& & &  & & 0\\
\end{pmatrix}
\qquad
\wh\Sigma^{(i)}= 
\begin{pmatrix}
0 & & & & & \\
& \ddots & & & &\\
& & 0 & & &\\
& & & \sigma_{i} & & \\
& & &  & \ddots & \\
& & &  & & \sigma_n\\
\end{pmatrix}.
\] 
So we get that for all $i=1,\dots,n+1$ it holds
$
A = \wt A^{(i)} + \wh A^{(i)} := U\wt\Sigma^{(i)} V + U\wh\Sigma^{(i)} V
$
and
\[
p(A) = \min_{i=1,\dots,n+1}\left\{ \frac{i-1}{n} + \sigma_i(A) \right\} = \min_{i=1,\dots,n+1}\left\{ \frac{\rk(\wt A^{(i)})}{n} + \|\wh A^{(i)}\| \right\}
\]
When $\serie A$ is a sequence with singular value symbol zero, and in this case we call it \textit{zero-distributed}, Theorem 3.2 in \cite{GS} says something more on the sequences $\wt A_n^{(i_n)}$ and $\wh A_n^{(i_n)}$, where $i_n$ are the indexes that minimize 
\[
 \frac{\rk(\wt A_n^{(i)})}{n} + \|\wh A_n^{(i)}\|.
\]

\begin{lemma}\label{spac}
Let $\serie A$ be a matrix sequence, and let $\wt A_n:=\wt A_n^{(i_n)}$, $\bar A_n:=\bar A_n^{(i_n)}$ be the sequences that minimize 
\[
 \frac{\rk(\wt A_n^{(i)})}{n} + \|\wh A_n^{(i)}\|.
\]
The following statement are equivalent.
\begin{itemize}
\item 
$\hfill \displaystyle \serie A\sim_\sigma 0 \hfill$
\item $\hfill \displaystyle 
\rk(\wt A_n) = o(n), \qquad \|\wh A_n\| = o(1)
\hfill$
\item $\hfill \displaystyle
\lim_{n\to\infty} \frac{\#\{ i: \sigma_i(A_n)>\ve \}}{n} = 0 \qquad \forall \ve >0
 \hfill$
\end{itemize}
\end{lemma}

The function $p(A)$ is subadditive, so we can introduce the pseudometric $d_{acs}$ on the space of matrix sequences
\[
\dacs{\serie{A}}{\serie{B}} = \limsup_{n\to \infty} p(A_n-B_n).
\]
It has been proved (\cite{bott},\cite{Garoni}) that this distance induces the a.c.s. convergence already introduced. In other words,
\[
\dacs{\serie A}{\B} \xrightarrow{m\to \infty} 0 \iff  \B\acs \serie A.
\]
This result holds since the function $p$ embodies the concept of acs convergence. In fact it is equivalent to
\[
p(A)= \inf \left\{ \frac{\rk(R)}{n} +\|N\|: \quad A = R+N  \right\}.
\]
This justifies the following result:
\begin{lemma}\label{spac2}
Let $\B$ and $\serie A$ be matrix sequences and let 
\[
R_{n,m} := \wt{A_n-B_{n,m}}^{(i_n)}, \qquad N_{n,m} := \wh{B_{n,m}-A_n}^{(i_n)},
\] 
where $i_n$ are the indexes that minimize 
\[
 \frac{\rk(\wt{ A_n - B_{n,m}}^{(i)})}{n} + \|\wh {A_n - B_{n,m}}^{(i)}\|.
\]
Then $\B\acs \serie A$ if and only if for every $m$ there exists $n_m$ with
\[
A_n = B_{n,m}  + N_{n,m} + R_{n,m}, \qquad \|N_{n,m}\|\le \omega(m), \qquad \rk(R_{n,m})\le nc(m)
\]
for every $n>n_m$, and
\[
\omega(m)\xrightarrow{m\to \infty} 0,\qquad c(m)\xrightarrow{m\to \infty} 0.
\]
\end{lemma}
\begin{proof}
Notice that 
\begin{align*}
\B\acs \serie A &\iff \dacs{\serie A}{\B} \xrightarrow{m\to \infty} 0 \\
&\iff  \limsup_{n\to \infty} p(A_n-B_{n,m})  \xrightarrow{m\to \infty} 0\\
& \iff \limsup_{n\to \infty} p(N_{n,m} + R_{n,m})  \xrightarrow{m\to \infty} 0\\
& \iff  \limsup_{n\to \infty}\|N_{n,m}\| + \frac{\rk(R_{n,m})}{n}  \xrightarrow{m\to \infty} 0
\end{align*}
\[ 
\iff   \limsup_{n\to \infty}\|N_{n,m}\| \xrightarrow{m\to \infty} 0, \qquad  \limsup_{n\to \infty}  \frac{\rk(R_{n,m})}{n} \xrightarrow{m\to \infty} 0.
 \]
The last two conditions are true if and only if for every $m$ there exists $n_m$ with
\[
\|N_{n,m}\|\le \omega(m), \qquad \rk(R_{n,m})\le nc(m) \qquad \forall n>n_m
\]
where
\[
\omega(m)\xrightarrow{m\to \infty} 0,\qquad c(m)\xrightarrow{m\to \infty} 0.
\]
\end{proof}

Notice that if $A_n$ is a diagonal matrix, then   $\wt A_n^{(i_n)}$ and $\wh A_n^{(i_n)}$ are also diagonal matrices. If a diagonal sequence $\serie D$ is zero distributed, thanks to Lemma \ref{spac}, we can find $\wt D_n^{(i_n)}$ and $\wh D_n^{(i_n)}$ respectively low rank and low norm diagonal matrices that sums up to $D_n$. Moreover, using Lemma \ref{spac2}, given $\{\{D_{n,m}\}_n\}_m\acs \serie D$ diagonal sequences we can find $R_{n,m}$ and $N_{n,m}$ respectively low rank and low norm diagonal matrices such that $D_n- D_{n,m} = N_{n,m} + R_{n,m}$.

$ $\\
In \cite{Mio}, has been proved that the pseudometric $d_{acs}$ on the space of matrix sequences is complete. Using the same arguments, we can prove the following more detailed statement.
\begin{theorem}\label{complete}
Let $\B$ be a sequence of matrix sequences that is a Cauchy sequence with respect to the pseudometric $d_{acs}$. There exists a crescent map $m:\f N\to \f N$ with $\lim_{n\to\infty} m(n) = \infty$ such that for every crescent map $m':\f N\to \f N$ that respects
\begin{itemize}
\item $m'(n)\le m(n) \quad \forall n $
\item $\lim_{n\to\infty} m'(n) = \infty$
\end{itemize}
we get
\[
\B \acs \{ B_{n,m'(n)} \}_n.
\]
In particular, the pseudometric $d_{acs}$ is complete.
\end{theorem}
\begin{proof}
By definition of Cauchy sequence, for every integer $k>0$, there exists an index $M_k$ such that
\[ 
\dacs{\{B_{n,s}\}_n}{\{B_{n,t}\}_n}\le 2^{-k} \qquad  \forall s,t\ge M_k.
 \]
We can suppose that $M_k$ are strictly increasing accordingly to $k$. Let us fix $k>0$ and consider all the couple of distinct sequences $(\{B_{n,s}\}_n,\{B_{n,t}\}_n)$ with $ M_{k+1} \ge s,t\ge M_k$. Their  is less then $2^{-k}$, so we can use the definition of distance
\[
\dacs{\{B_{n,s}\}_n}{\{B_{n,t}\}_n} = \limsup_{n\to \infty} p(B_{n,s}-B_{n,t})\le 2^{-k}
\]
and obtain that there exists an index $N_{s,t}$ such that
\[
p(B_{n,s}-B_{n,t})\le 2^{1-k} \qquad \forall n \ge N_{s,t}.
\]
Since the number of couples $(s,t)$ between $M_k$ and $M_{k+1}$ is finite, we can define the minimum of $N_{s,t}$ as
\[ 
\wt N_k := \min_{M_{k+1} \ge s,t\ge M_k} N_{s,t} 
 \]
 and obtain a strictly increasing sequence $N_k$ by the following recursive definition
 \[
 N_1 = \wt N_1, \qquad N_k = \max\{\wt N_k, N_{k-1}+1 \} \quad \forall k>1.
 \]
Notice that, given $k>0$, $n\ge N_k$ and   $ M_{k+1} \ge s,t\ge M_k$ we have
\begin{equation}\label{one}
n\ge N_k\ge \wt N_k\ge N_{s,t} \implies   p(B_{n,s}-B_{n,t}) \le 2^{1-k}. 
\end{equation}
We can now define the function 
\[
m(n):=
\begin{cases}
1 & N_{1}>n \\
M_k & N_{k+1}>n\ge N_k \quad \forall k>0
\end{cases}
\]
The map $m(n)$ is surely increasing, and it diverges to infinite as $M_k$ are strictly increasing. 
Let $m':\f N\to \f N$ be a crescent map with
\begin{itemize}
\item $m'(n)\le m(n) \quad \forall n $
\item $\lim_{n\to\infty} m'(n) = \infty$
\end{itemize}
and set $A_n:= \{B_{n,m'(n)}\}_n$.
We can now prove that $\B$ converges acs to $\serie A$. 
Let $m$ be a fixed index with $M_{k}\le m < M_{k+1}$.
Let $n$ be a fixed index with $N_{k+1}\le N_s \le n < N_{s+1}$ and since $m'(n)$ is not bounded, we can also assume $m'(n)\ge m$.
There also exists a $t$ such that $M_{k}\le M_{t}\le m'(n) < M_{t+1}$, and 
\[
M_t\le m'(n)\le m(n) = M_s \implies t\le s.
\]
Since the function $p$ is subadditive, we get
\[
 p(A_{n}-B_{n,m}) = p(B_{n,m'(n)} - B_{n,m}) \le\]
 \[
 \le p(B_{n,m'(n)} - B_{n,M_{t}}) +  \sum_{r=k+1}^{t-1} p(B_{n,M_{r+1}} - B_{n,M_{r}}) + p(B_{n,M_{k+1}} - B_{n,m}).
\]
Using (\ref{one}), we know that
\[
n \ge N_s\ge N_t, \quad M_{t+1}> m'(n) \ge M_{t} \implies  p(B_{n,M_{k+1}} - B_{n,m}) \le 2^{1-t}
\]
\[
n\ge N_s\ge N_t > N_r \implies p(B_{n,M_{r+1}} - B_{n,M_{r}})\le 2^{1-r} \qquad \forall k< r<t
\]
\[
n\ge N_{k+1} >N_k, \quad M_{k+1}> m \ge M_{k}   \implies  p(B_{n,M_{k+1}} - B_{n,m}) \le 2^{1-k}
\]
so we get
\[
p(B_{n,m'(n)} - B_{n,M_{t}}) +  \sum_{r=k+1}^{t-1} p(B_{n,M_{r+1}} - B_{n,M_{r}}) + p(B_{n,M_{k+1}} - B_{n,m})\]
\[
\le 2^{1-t} + \sum_{r=k+1}^{t-1} 2^{1-r} + 2^{1-k}
= \sum_{r=k}^{t} 2^{1-r} \le \sum_{r=k}^{\infty} 2^{1-r} = 4 \cdot 2^{-k}.
\]
This means that for all $m$ with $M_{k}\le m < M_{k+1}$ and for all $n$ with $N_{k+1}\le  n$ and such that $m'(n)\ge m$, we get
\[
p(A_{n}-B_{n,m}) \le 4 \cdot 2^{-k}
\]
so
\[
\dacs{\serie A}{\{B_{n,m}\}_n} = \limsup_{n\to \infty} p(A_n-B_{n,m}) \le 4 \cdot 2^{-k} \qquad \forall M_{k}\le m < M_{k+1}
\]
\[
\implies \limsup_{m\to \infty} \dacs{\serie A}{\{B_{n,m}\}_n} \le \limsup_{k\to \infty}4 \cdot 2^{-k} = 0.
\]
This entails that $\B\acs \serie A$.
\end{proof}

\subsection{Diagonal Sequences and Spectral Symbols}

The acs convergence has nice properties linked to the singular value symbols and GLT symbols of the sequences, as we will see in the next subsection. The spectral symbols behave nicely only on the sequences of hermitian matrices, but when we work with diagonal matrices, not necessarily with real entries, we can also regain some properties.

\begin{lemma}\label{lambdacomp}
Given any almost everywhere continuous function $a:[0,1] \to \f C$, then
\[ 
D_n(a)\sim_\lambda a(x).
 \]
\end{lemma}
\begin{proof}
Given any $G\in C_c(\f C)$, we know that
\[
\lim_{n\to \infty}\frac 1n \sum_{i=1}^n G\left(a\left(\frac in \right)\right) = \int_0^1 G(a(x)) dx 
\]
since $G\circ a:[0,1]\to \f C$ is a bounded and almost everywhere continuous function, thus Riemann Integrable.
\end{proof}

\begin{lemma}\label{zerosum}
Let $\serie D \sim_\lambda f(x)$ and $\serie Z\sim_\sigma 0$ be diagonal matrices sequences, where $f$ is any measurable function on a domain $D\cu \f R^n$ with finite non-zero measure. Then
\[ 
\serie D+\serie Z\sim_\lambda f(x)
 \]
 \end{lemma}
\begin{proof}
Using Lemma \ref{spac}, we know that given any $\ve > 0$ we have
\[
\lim_{n\to\infty} \frac{\#\{ i: \sigma_i(Z_n)>\ve \}}{n} = 0.
\]
Let the singular values of $D_n$ and $Z_n$ be sorted such that the $i$-th singular value is the absolute value of the $i$-th entry on their diagonal. 
\[
\sigma_i(D_n) = \Big|[D_n]_{i,i}\Big|,\qquad \sigma_i(Z_n) = \Big|[Z_n]_{i,i}\Big|,\qquad  \sigma_i(D_n+Z_n) = \Big|[D_n+Z_n]_{i,i}\Big|.
\]
We thus have
\[
\sigma_i(D_n+Z_n) - \sigma_i(D_n) \le \sigma_i(Z_n).
\]
Given any $G\in C_c(\f C)$, we get
\begin{align*}
\Bigg|\frac 1n \sum_{i=1}^n  G(\sigma_i(D_n+Z_n)) &- \int_D G(f(x)) dx\Bigg| \\  
\le &\left|\frac 1n \sum_{i=1}^n G(\sigma_i(D_n+Z_n)) - \frac 1n \sum_{i=1}^n G(\sigma_i(D_n)) \right|  \\
 & + \left|\frac 1n \sum_{i=1}^n G(\sigma_i(D_n)) - \int_D G(f(x)) dx\right| 
\end{align*}
where the first term of the sum can be written as
\begin{align*}
\Bigg|\frac 1n \sum_{i=1}^n  G(\sigma_i(D_n+Z_n)) &-  \frac 1n \sum_{i=1}^n G(\sigma_i(D_n)) dx\Bigg| \\  
\le &\left|\frac 1n \sum_{\sigma_i(Z_n)>\ve} G(\sigma_i(D_n+Z_n)) - \frac 1n \sum_{i=1}^n G(\sigma_i(D_n)) \right|  \\
 & + \left|\frac 1n \sum_{\sigma_i(Z_n)\le\ve} G(\sigma_i(D_n+Z_n)) - \frac 1n \sum_{i=1}^n G(\sigma_i(D_n)) \right|\\
 \le &\,\, 2\|G\|_\infty\frac{\#\{ i: \sigma_i(Z_n)>\ve \}}{n} + \omega_G(\ve)
\end{align*}
and the second term of the sum tends to zero when $n$ goes to infinity. This leads to
\[ 
\limsup_{n\to \infty}  \Bigg|\frac 1n \sum_{i=1}^n  G(\sigma_i(D_n+Z_n)) - \int_D G(f(x)) dx\Bigg|  \le \omega_G(\ve) \qquad \forall \ve >0
\]
but $\omega_G(\ve)$ is the modulus of continuity of $G$ and tends to zero as $\ve$ goes to zero, so we get the thesis
 \[ 
 \lim_{n\to \infty}  \frac 1n \sum_{i=1}^n  G(\sigma_i(D_n+Z_n)) = \int_D G(f(x)) \implies \serie D+\serie Z\sim_\lambda f(x)
 \]
\end{proof}

The space of diagonal sequences that admit a spectral symbol possess also a closure property with respect to the acs convergence of sequences and the convergence in measure of measurable functions.

\begin{lemma}\label{commuta}
Let $\serie D$ and $\{\{D_{n,m}\}_n\}_m$ be diagonal matrices sequences, and let $a:D\to \f C$, $a_m:D\to \f C $ be measurable functions defined on $D\cu \f R^n$ that is a measurable set with non zero finite measure. If    
\begin{itemize}
\item $\{D_{n,m}\}_{n,m}\sim_\lambda a_m(x)$
\item $\{D_{n,m}\}_{n,m}\acs \serie D$
\item $a_m(x)\xrightarrow{\mu} a(x)$
\end{itemize}
then $\serie D\sim_\lambda a(x)$.
\end{lemma}
\begin{proof}
Let $G\in C_c(\f C)$, and let the eigenvalues of the diagonal matrices be sorted so that the $i$-th eigenvalue corresponds to the $i$-th diagonal element.
\[ \lambda_i(D_n) =  [D_n]_{i,i}, \qquad \lambda_i(D_{n,m}) = [D_{n,m}]_{i,i} . \] 
We know that
\begin{align*}
\Bigg|\frac 1n \sum_{i=1}^n &G(\lambda_i(D_n)) - \int_D G(a(x)) dx\Bigg| \le\\
&\left|\frac 1n \sum_{i=1}^n G(\lambda_i(D_n)) - \frac 1n \sum_{i=1}^n G(\lambda_i(D_{n,m}))\right|\\ 
&+ \left|\frac 1n \sum_{i=1}^n G(\lambda_i(D_{n,m})) - \int_D G(a_m(x)) dx\right|\\
&+ \left|\int_D G(a_m(x)) dx - \int_D G(a(x)) dx\right|.
\end{align*}
Using Lemma \ref{spac2}, $D_n = D_{n,m} + R_{n,m} + N_{n,m}$ where $R_{n,m}$ and $N_{n,m}$ are diagonal matrices and for every $m$ there exists $n_m$ with
\[
\rk (R_{n,m})\le c(m) n, \qquad \|N_{n,m}\| \le s(m) \qquad \forall n>n_m
\]
\[ \lim_{m\to \infty} s(m) = \lim_{m\to \infty} c(m)=0. \]
Moreover, there exists a function $t(m)$ such that 
\[
\left|\int_D G(a_m(x)) dx - \int_D G(a(x)) dx\right| \le t(m) \qquad \lim_{m\to \infty}t(m)= 0
  \]
So we can fix $\ve >0$ and find an index $M$ such that 
\[
c(m)\le \ve, \quad  s(m) \le \ve, \quad t(m)\le \ve \qquad \quad \forall m>M. 
\]
We know that
\[
\lambda_i(D_n) - \lambda_i(D_{n,m})= \lambda_i(D_n-D_{n,m}) = \lambda_i(N_{n,m}+R_{n,m}) = \lambda_i(N_{n,m})+\lambda_i(R_{n,m})
\]
but if $n>n_m$ and $m>M$, then
\begin{align*}
\ve < |\lambda_i(D_n-D_{n,m})| & \implies \ve < |\lambda_i(N_{n,m})|+|\lambda_i(R_{n,m})|\le \ve + |\lambda_i(R_{n,m})|\\
& \implies 0 \ne \lambda_i(R_{n,m}).
\end{align*}
This means that 
\[
\#\{i: \ve < \lambda_i(D_n-D_{n,m}) \} \le \rk(R_{n,m})
\]
so we can estimate
\begin{align*}
\Bigg|\frac 1n \sum_{i=1}^n &G(\lambda_i(D_n)) - \frac 1n \sum_{i=1}^n G(\lambda_i(D_{n,m}))\Bigg| \le\\
&\left|\frac 1n \sum_{i:  \lambda_i(D_n-D_{n,m})\le \ve} G(\lambda_i(D_n)) - G(\lambda_i(D_{n,m}))\right|\\
&+ \left|\frac 1n \sum_{i: \lambda_i(D_n-D_{n,m})>\ve } G(\lambda_i(D_n)) -  G(\lambda_i(D_{n,m}))\right|\\
&\le \omega_G(\ve) + 2\frac{rk(R_{n,m})}{n}\|G\|_{\infty}\\
&\le \omega_G(\ve) + 2\ve\|G\|_{\infty}\qquad \forall n>n_m,\quad \forall m>M.
\end{align*}
We obtain
\begin{align*}
\limsup_{n\to\infty}\Bigg|\frac 1n \sum_{i=1}^n &G(\lambda_i(D_n)) - \int_D G(a(x)) dx\Bigg| \le\\
&\limsup_{n\to\infty}\left|\frac 1n \sum_{i=1}^n G(\lambda_i(D_n)) - \frac 1n \sum_{i=1}^n G(\lambda_i(D_{n,m}))\right|\\ 
&+ \limsup_{n\to\infty}\left|\frac 1n \sum_{i=1}^n G(\lambda_i(D_{n,m})) - \int_D G(a_m(x)) dx\right|\\
&+ \limsup_{n\to\infty}\left|\int_D G(a_m(x)) dx - \int_D G(a(x)) dx\right|\\
&\le \omega_G(\ve) + 2\ve\|G\|_{\infty} + \ve \qquad \forall \ve >0
\end{align*}
and this concludes
\[
\lim_{n\to\infty} \frac 1n \sum_{i=1}^n G(\lambda_i(D_n)) = \int_D G(a(x))\implies  \serie D\sim_\lambda a(x).\]
\end{proof}

\subsection{GLT}

A matrix sequence $\serie A$ may have several different singular values symbols, even on the same domain. For specific sequences we can choose one of their symbols, and denote it as \textit{GLT symbol} of the sequence
\[
\serie A\GLT k(x,\theta).
\]
If we call $\mc G$ the set of sequences that own a GLT symbol, we know that
\begin{itemize}
\item $\mc G$ is a $\f C$ algebra. In fact, given $\serie A,\serie B\in \mc G$ and $c\in \f C$, we have that
\[ 
\{ A_n+B_n\}_n\in \mc G \qquad \{A_nB_n\}_n\in \mc G\qquad \{cA_n\}_n\in\mc G.
 \]
\item $\mc G$ is closed with respect to the acs pseudometric. If $\{B_{n,m}\}_m\in\mc G$ for all $m$, then
\[
\B \acs \serie A\implies \serie A\in \mc G
\]
\end{itemize} 

The chosen symbols have all the same domain $D=[0,1]\times[-\pi,\pi]$. Let $\mc M_D$ be the set of measurable functions on $D$, where we identify two functions if they are equal almost everywhere. The choice of the symbol can be seen as a map
\[
S : \mc G \to \mc M_D. 
\] 
if we endow $\mc M_D$ with the metric $d_M$ that induces the convergence in measure, then $S$ can be seen as a map of algebras and complete pseudometric spaces, with the following properties that are proved in \cite{GS} or in \cite{Mio}.
\begin{theorem}\label{S}$ $
\begin{enumerate}
\item $S$ is an homomorphism of algebras. Given $\serie A,\serie B\in \mc G$ and $c\in \f C$, we have that
\[ 
S(\{ A_n+B_n\}_n) = S(\serie A+\serie B),\]\[ S(\{A_nB_n\}_n) = S(\serie A)S(\serie B),\]\[ S(\{cA_n\}_n)= cS(\serie A).
 \]
\item The kernel of $S$ are exactly the zero-distributed sequences.
\item $S$ preserves the distances. Given $\serie A,\serie B\in \mc G$ we have
\[
\dacs{\serie A}{\serie B} = d_m(S(\serie A),S(\serie B)).
\]
\item $S$ is onto. All measurable functions are GLT symbols.
\item GLT symbols are spectral symbols:
\[
\serie A\in \mc G\implies \serie A\sim_\sigma S(\serie A)
\]
\item The graph of $S$ is closed in $\mc G\times \mc M_D$. If $\B$ are sequences in $\mc G$ that converge acs to $\serie A$, and their symbols converge in measure to $k(x,\theta)$, then $S(\serie A) = k(x,\theta)$.
\end{enumerate}
\end{theorem}

\section{Piecewise Convergence}

Let us formalize the concept of piecewise convergence shown in the introduction.\\
Given a vector $v\in \f C^n$, let $f:[0,1]\to \f C$ be the piecewise linear function that interpolates $v_i$ on the points $i/n$ of the interval $[0,1]$, where $v_0:=0$. It respects
\[
f\left(\frac{i}{n}\right) = v_i \qquad \forall \,i=0,1,\dots,n
\]
and its analytical expression is
\[
f(x)  = 
nx(v_{i}^{(n)}-v_{i-1}^{(n)}) + i(v_{i-1}^{(n)}-v_{i}^{(n)}) + v_{i}^{(n)}  \quad  x\in \left[\frac {i-1}n,\frac in\right ] 
\]
 
\begin{definition}
Given a sequence  $\serie D$ of diagonal matrices, let $d_n = \text{diag}(D_n)$, and $f_n$ the maps associated to $d_n$. We say that $D_n$ converges \textit{piecewise} to a function $f:[0,1]\to \f C$ if the sequence $f_n$ converges to $f$ in measure. In this case we write
\[
\serie D \tratti f(x).
\]
\end{definition}

This definition copes well with the concept of zero distributed sequences and diagonal sampling of continuous functions.

\subsection{Zero Distributed and Diagonal Sampling Sequences}

We recall here that a sequence is called zero distributed if the function zero is a singular value symbol. Moreover, they are all GLT sequences with symbol zero, and coincide with the kernel of the map $S$ in Theorem \ref{S}, that is, are the only GLT sequences with symbol zero.
In the case of diagonal matrices, we can prove that a diagonal sequence is zero-distributed if and only if it converges piecewise to the function zero.

\begin{lemma}\label{zerotratti}
Given a sequence of diagonal matrices $\serie Z$, the following statement are equivalent.
\begin{enumerate}
\item $\serie Z\tratti 0$
\item $\serie Z\GLT 0$
\item $\serie Z\sim_\lambda 0$
\end{enumerate}
\end{lemma}
\begin{proof} A matrix sequence is a GLT sequence with symbol zero if and only if it is a zero distributed sequence, so we can substitute the second statement with $\serie Z\sim_\sigma 0$.\\

\noindent$2\implies 3)$ Using Lemma \ref{spac}, we can split the eigenvalues of $Z_n$ into two diagonal matrices $Z_n = \wh Z_n + \wt Z_n$ with $\rk(\wh Z_n) = o(n)$ and $\|\wt Z_n\| = o(1)$. This means that given $\ve$, there exists $N$ such that 
\[
\rk(\wh Z_n)<\ve n, \qquad \|\wt Z_n\|<\ve \qquad \forall n\ge N
\]
or also said as
\[
\#\{i :  |\lambda_i(Z_n)|\ge \ve \} < \ve n \qquad \forall n\ge N. 
\]
We can now use the definition of spectral symbol to prove the thesis. Given any $F\in C_c(\f C)$, we get
\[
\left|\frac{1}{n} \sum_{i=1}^{n} F(\lambda_i(Z_n)) - F(0)\right| \le\]\[
\left|\frac{1}{n} \sum_{|\lambda_i(Z_n)|\ge \ve} F(\lambda_i(Z_n))\right|  +  \left|\frac{1}{n} \sum_{|\lambda_i(Z_n)|< \ve} F(\lambda_i(Z_n)) - F(0)\right|\le 
\]
\[
\frac{\#\{i :  |\lambda_i(Z_n)|\ge \ve \}}{n} \|F\|_\infty  +  \omega_F(\ve) \le \ve \|F\|_\infty + \omega_F(\ve)\qquad \forall n\ge N
\]
where $\omega_F(\ve)$ is the modulus of continuity of the function $F$. We thus get
\[
\limsup_{n\to \infty}\left|\frac{1}{n} \sum_{i=1}^{n} F(\lambda_i(Z_n)) - F(0)\right| \le \ve \|F\|_\infty + \omega_F(\ve)
\]
for every $\ve >0$. So we get to the conclusion
\[
\lim_{n\to \infty} \frac{1}{n} \sum_{i=1}^{n} F(\lambda_i(Z_n)) = F(0) = \int_0^1 F(0) dx
\]
\[
\implies \serie Z \sim_\lambda 0
\]

\noindent$2\impliedby 3)$  Let $F\in C_c(\f R)$. If $G:\f C\to \f C$ is defined as $G(z):= F(|z|)$, we get that this is still a continuous and compact supported function $G\in C_c(\f C)$. We also know that the singular values of a diagonal matrix are the absolute values of the eigenvalues. With this we can conclude
\[
\lim_{n\to \infty} \frac{1}{n} \sum_{i=1}^{n} F(\sigma_i(Z_n)) =
\lim_{n\to \infty} \frac{1}{n} \sum_{i=1}^{n} F(|\lambda_i(Z_n)|) =
\lim_{n\to \infty} \frac{1}{n} \sum_{i=1}^{n} G(\lambda_i(Z_n)) 
\]
\[ =\int_0^1 G(0) dx  = \int_0^1 F(0) dx
\implies \serie Z \sim_\sigma 0
\]

\noindent$1\implies 2)$ Let $f_n(x)$ be the piecewise linear functions associated with $Z_n$. We fix $\ve >0$ and study the quantity
\[
E_{i,n} := n\cdot \mu\left\{x\in \left[\frac{i-1}n,\frac in \right] : |f_n(x)|>\ve \right\}
\]
that we can divide into $E_{i,n} = E_{i,n}^++E_{i,n}^-$ where 
\[
E_{i,n}^+ := n\cdot \mu\left\{x\in \left[\frac{i-1}n,\frac in \right] : f_n(x)>\ve \right\} 
\]
\[
E_{i,n}^- := n\cdot \mu\left\{x\in \left[\frac{i-1}n,\frac in \right] : f_n(x)<-\ve \right\}.
\]
Suppose now that $f_n\left(\frac{i}{n}\right) >2\ve$. Recall that 
\[
f_n(x) =
nx\left[f_n\left(\frac in\right)-f_n\left(\frac {i-1}n\right)\right] + i\left[f_n\left(\frac {i-1}n\right)-f_n\left(\frac in\right)\right] + f_n\left(\frac in\right).
\]
Let us divide the analysis in cases.
\begin{itemize}
\item If $ f_n\left(\frac{i-1}{n}\right)\ge \ve$, then $E_{i,n}=1$. 
\item If $ |f_n\left(\frac{i-1}{n}\right)|< \ve $, then $E_{i,n}^-=0$, so $E_{i,n}^+=E_{i,n}$. Given $ x\in \left[\frac{i-1}n,\frac in \right] $ we have
\[
f_n(x)>\ve \iff  \frac in >x >
\frac{
\ve -i(f_n\left(\frac {i-1}n\right)-f_n\left(\frac in\right)) - f_n\left(\frac in\right)
}{
n(f_n\left(\frac in\right)-f_n\left(\frac {i-1}n\right))
}
\]
so
\[
E_{i,n}^+ =  n\left[\frac in - \frac{
 \ve -i(f_n\left(\frac {i-1}n\right)-f_n\left(\frac in\right)) - f_n\left(\frac in\right)
 }{
 n(f_n\left(\frac in\right)-f_n\left(\frac {i-1}n\right))
 }\right]
= 
\frac{ 
 f_n\left(\frac in\right) -\ve
 }{
 f_n\left(\frac in\right)-f_n\left(\frac {i-1}n\right)
 }.
\]
Using the hypothesis, we get
\[
E_{i,n} = 
\frac{ 
 f_n\left(\frac in\right) -\ve
 }{
 f_n\left(\frac in\right)-f_n\left(\frac {i-1}n\right)
 }
 \ge
\frac{ 
 f_n\left(\frac in\right) -\ve
 }{
 f_n\left(\frac in\right)+\ve
 }  = 
1
 -
 \frac{ 
2\ve
  }{
 f_n\left(\frac in\right)+\ve
  }
  \ge 
  1
   -
   \frac{ 
  2\ve
    }{
   2\ve+\ve
    }
  = \frac 13.
\]
\item If $ f_n\left(\frac{i-1}{n}\right)\le -\ve $, then given $ x\in \left[\frac{i-1}n,\frac in \right] $ we have
\[
f_n(x)>\ve \iff  \frac in >x >
\frac{
\ve -i(f_n\left(\frac {i-1}n\right)-f_n\left(\frac in\right)) - f_n\left(\frac in\right)
}{
n(f_n\left(\frac in\right)-f_n\left(\frac {i-1}n\right))
},
\]
\[
f_n(x)<-\ve \iff  \frac {i-1}n <x <
\frac{
-\ve -i(f_n\left(\frac {i-1}n\right)-f_n\left(\frac in\right)) - f_n\left(\frac in\right)
}{
n(f_n\left(\frac in\right)-f_n\left(\frac {i-1}n\right))
},
\]
so
\[
E_{i,n}^+ =  n\left[\frac in - \frac{
 \ve -i(f_n\left(\frac {i-1}n\right)-f_n\left(\frac in\right)) - f_n\left(\frac in\right)
 }{
 n(f_n\left(\frac in\right)-f_n\left(\frac {i-1}n\right))
 }\right]
= 
\frac{ 
 f_n\left(\frac in\right) -\ve
 }{
 f_n\left(\frac in\right)-f_n\left(\frac {i-1}n\right)
 },
\]
\[
E_{i,n}^- =  n\left[ \frac{
 - \ve -i(f_n\left(\frac {i-1}n\right)-f_n\left(\frac in\right)) - f_n\left(\frac in\right)
 }{
 n(f_n\left(\frac in\right)-f_n\left(\frac {i-1}n\right))
 }
 - \frac{i-1}n\right]
= 
\frac{ 
 - \ve -f_n\left(\frac {i-1}n\right)
 }{
 f_n\left(\frac in\right)-f_n\left(\frac {i-1}n\right)
 },
\]
\[
E_{i,n} = \frac{ 
 f_n\left(\frac in\right) -\ve
 }{
 f_n\left(\frac in\right)-f_n\left(\frac {i-1}n\right)
 }
 +
 \frac{ 
  - \ve -f_n\left(\frac {i-1}n\right)
  }{
  f_n\left(\frac in\right)-f_n\left(\frac {i-1}n\right)
  }
  = 
  1 - \frac{2\ve}{f_n\left(\frac in\right)-f_n\left(\frac {i-1}n\right)}.
\]
Using the hypothesis, we get
\[
E_{i,n} = 
1 - \frac{2\ve}{f_n\left(\frac in\right)-f_n\left(\frac {i-1}n\right)}
 \ge 
1 - \frac{2\ve}{2\ve+\ve}
= \frac 13
\]
\end{itemize}
The same would happen if $f_n(i/n)<-2\ve$, so
\begin{align*}
\mu \{ x: |f_n(x)|>\ve  \} 
& = \sum_{i=1}^n\mu \left\{ x: |f_n(x)|>\ve,\,\, x\in \left[\frac{i-1}n,\frac in \right]  \right\}\\
& \ge \sum_{i: |f_n(i/n)|>2\ve}\mu \left\{ x: |f_n(x)|>\ve,\,\, x\in \left[\frac{i-1}n,\frac in \right]  \right\}\\
& = \sum_{i: |f_n(i/n)|>2\ve}\frac 1n E_{i,n}\ge \frac 1{3n} \#\left\{ i: \left|f_n\left(\frac in\right)\right|>2\ve \right\}.
\end{align*}
We also know that $f_n\xrightarrow{\mu} 0$, so 
\[
\lim_{n\to\infty} \mu \{ x: |f_n(x)|>\ve  \} = 0 \qquad \forall \ve >0
\]
that leads to
\[
\lim_{n\to\infty}\frac 1{3n} \# \left\{ i: \left|f_n\left(\frac in\right)\right|>2\ve \right\} = 0 \qquad \forall \ve >0
\]
The values $\left|f_n\left(\frac in\right)\right|$ are the singular values of $Z_n$ for every $i>0$, so we get also
\[
\lim_{n\to\infty}\frac 1{3n} \# \left\{ i: \sigma_i(Z_n)>2\ve \right\} = 0 \qquad \forall \ve >0
\]
and thanks to Lemma \ref{spac}, we can conclude that $\serie Z\sim_\sigma 0$.\\

\noindent$1 \impliedby 2)$ Let $\ve>0$ be a fixed value. We notice that
\[
|f_n(x)|\ge \ve\implies \max\left\{ \left| f_n\left(\frac {i-1}n\right)\right|,\left| f_n\left(\frac in\right)\right|\right\}\ge \ve \qquad \forall  x\in \left[\frac{i-1}n,\frac in \right]
\]
so we deduce
\begin{align*}
\mu \{ x: |f_n(x)|>\ve  \} 
& = \sum_{i=1}^n\mu \left\{ x: |f_n(x)|>\ve,\,\, x\in \left[\frac{i-1}n,\frac in \right]  \right\}\\
& = \sum_{i:  \max\left\{ \left| f_n\left(\frac {i-1}n\right)\right|,\left| f_n\left(\frac in\right)\right|\right\}\ge \ve}\mu \left\{ x: |f_n(x)|>\ve,\,\, x\in \left[\frac{i-1}n,\frac in \right]  \right\}\\
& \le \frac 1n  \# \left\{ i: \max\left\{ \left| f_n\left(\frac {i-1}n\right)\right|,\left| f_n\left(\frac in\right)\right|\right\}\ge \ve   \right\}    \\
& \le \frac 2n  \#\left\{ i:  \left| f_n\left(\frac {i}n\right)\right|\ge \ve   \right\}  .
\end{align*}
The values $\left|f_n\left(\frac in\right)\right|$ are the singular values of $Z_n$ for every $i>0$, so we can use Lemma \ref{spac}, and obtain that the last quantity converges to zero when $n$ tends to infinity. Therefore,
\[
\lim_{n\to \infty} \mu \{ x: |f_n(x)|>\ve  \}  = 0 \qquad \forall \ve >0
\]
meaning that $f_n\xrightarrow{\mu} 0$.
\end{proof}

Given the sampling diagonal matrices $\sdiag{a}$, where $a:[0,1]\to \f C$ is a continuous function, and
\[
D_n(a):= \underset{i=1,\dots,n}{\text{diag}} a\left(\frac{i}{n}\right) = 
\begin{pmatrix}
a\left(\frac{1}{n}\right) & & & \\
& a\left(\frac{2}{n}\right) & & \\
& & \ddots & \\
& & & a(1)
\end{pmatrix}
\]
we can prove that $\sdiag a\tratti a(x)$.

\begin{lemma}\label{diagcont}
Given any continuous map $a:[0,1]\to\f C$, the functions $f_n$ associated to $D_n(a)$ converge in measure to $a$. In particular,
	\[
	\sdiag a \tratti a(x).
	\] 
\end{lemma}
\begin{proof}
Any continuous function on $[0,1]$ is uniformly continuous thanks to Heine-Cantor theorem, so given any $\ve>0$, there exists $\delta>0$ such that 
\[
|x-y|\le\delta\implies |a(x)-a(y)|\le \ve.
\]
Let $n$ be a natural number with $n^{-1}<\delta$. If we denote $a_k:=a(k/n)$, then 
\[ 
|a(x) - a_i| \le \ve \qquad \forall x\in \left[ \frac{i-1}{n},\frac in \right]\quad \forall i.  
 \]
On the same interval, the function $f_n$  is a segment in the complex space from $a_{i-1}$ to $a_i$ for every $i>1$. We know that $|a_i-a_{i-1}|\le \ve$ , so
\[ 
|f_n(x) - a_i| \le \ve \qquad \forall x\in \left[ \frac{i-1}{n},\frac in \right] \quad \forall i>1.  
 \]
This entails that
\[ |f_n(x)-a(x)|\le 2\ve  \qquad \forall x\in \left[\frac 1n,1\right] \quad \forall n > \delta^{-1}. \]
So the functions $f_n(x)$ converge in measure to $a(x)$ and
\[
\sdiag a\tratti a(x)
\] 
\end{proof}
The set of continuous functions is dense in the space of measurable functions, so we can prove some approximation results on convergent sequences of continuous functions.

\begin{lemma}\label{mn}
Given any $a:[0,1]\to\f C$ measurable function, and $a_m\in C([0,1])$ continuous functions that converge in measure to $a(x)$, there exists a crescent and unbounded map $m(n)$ such that
\[
\{D_n(a_{m(n)})\}_n\GLT a(x)\otimes 1\qquad \sdiag{a_{m(n)}}\tratti a(x)
\]
\end{lemma}
\begin{proof}
The functions $a_m(x)$ are continuous, so $\sdiag{a_m}\GLT a_m(x)\otimes 1$, and $a_m(x)\otimes 1$ converges in measure to $a(x)\otimes 1$. This means that the sequence $a_m(x)\otimes 1$ is a Cauchy sequence in $\mc M_D$, but the map $S$ of Theorem \ref{S} preserves the distances, so $\{\sdiag{a_m}\}_m$ is also a Cauchy sequence. Thanks to Theorem \ref{complete}, we know that there exists an unbounded crescent map $m_1(n)$ such that
\[ \{\sdiag{a_m}\}_m \acs  \sdiag{a_{m_1(n)}}.\]  
The map $S$ has also a closed graph, so we get
\[\sdiag{a_{m_1(n)}}\GLT a(x)\otimes 1.\]  
Thanks to Lemma \ref{diagcont}, we also know that $\sdiag {a_m}\tratti a_m(x)$ for every $m$, meaning that the piecewise linear functions $a_{n,m}(x)$ associated to $D_n(a_m)$ converge to $a_m(x)$.
The measure convergence is metrizable through the distance $d_M$, so we can define the following indexes.
\begin{itemize}
\item $a_m(x)$ converges to $a(x)$ in measure, so for every $k>0$ there exists an index $M_k$ such that
\[
M_k< M_{k+1}, \qquad d_M(a_m(x),a(x)) \le 2^{-k} \qquad \forall m\ge M_k\quad \forall k>0
\]
\item $a_{n,m}(x)$ converges to $a_{m}(x)$ in measure, so there exists an index $\wt N_m$ such that
\[
d_M(a_{n,m}(x),a_{m}(x)) \le 2^{-k} \qquad \forall n\ge \wt N_m
\]
and we can define a strictly increasing set of indexes $N_k$ with the following recursive procedure.
\[
N_1 = \max_{m\le M_1} \left\{\wt N_m\right\}
 \qquad 
N_{k} = \max\left\{
\max_{M_k\le m\le M_{k+1}}
\{\wt N_m\}, 
N_{k-1} +1\right\} 
\quad \forall k>1
\]
\end{itemize}
We can now define a new crescent unbounded map $m_2(n)$ 
\[
m_2(n):=
\begin{cases}
1 & N_{1}>n \\
M_k & N_{k+1}>n\ge N_k \quad \forall k>0
\end{cases}
\]
and prove that the sequence $a_{n,m'(n)}(x)$ converges to $a(x)$ for any crescent unbounded map $m'(n)$ such that $m'(n)\le m_2(n)$. In fact, suppose that $n\ge N_k$ and $m'(n)\ge M_k$. We get
\[
d_M(a_{n,m'(n)}(x), a(x))  \le  d_M(a_{n,m'(n)}(x), a_{m'(n)}(x)) + d_M(a_{m'(n)}(x), a(x)) 
\]
and
\[
m'(n) \ge M_k \implies d_M(a_{m'(n)}(x), a(x)) \le 2^{-k}
\]
\[
M_k \le M_s \le m'(n) < M_{s+1}, \quad m'(n)\le m_2(n) \]
\[
\implies M_k\le M_s\le m_2(n), \quad k\le s,\quad  n \ge N_k\]\[ 
\implies  d_M(a_{n,m'(n)}(x), a_{m'(n)}(x)) \le 2^{-s}\le 2^{-k} .
\]
Consequentially,
\[
d_M(a_{n,m'(n)}(x), a(x))  \le 2^{1-k}
\]
so the distance goes to zero when $n$ tends to infinity,and this implies that
\[
\sdiag{a_{m'(n)}} \tratti a(x)
\]
If we consider $m(n) := \min\{m_1(n),m_2(n)\}$ we find that it is an unlimited crescent function and it is bounded by both maps $m_1(n)$ and $m_2(n)$. Using again Lemma \ref{complete} and what we've shown before, we prove the thesis.
\end{proof}

\subsection{Main Results}

We are now ready to see how the piecewise convergence, the concept of spectral symbol and the GLT symbols are connected when dealing with diagonal sequences.

\begin{theorem}\label{trattilambda}
Given any diagonal sequence $\serie D$ and a measurable function $f:[0,1]\to \f C$, then
\[
 \serie D\tratti f(x)\implies \serie D\sim_\lambda f(x).
\]
\end{theorem}
\begin{proof}
Let $f_m(x)$ be continuous functions that converge in measure to $f(x)$. Using Lemma \ref{mn}, we can find a map $m(n)$ such that
\[
\sdiag{f_{m(n)}}\GLT f(x) \otimes 1, \qquad \sdiag{f_{m(n)}}\tratti f(x).
\]
We know that $\sdiag{f_m}\GLT f_m(x) \otimes 1$, but the map $S$ in Theorem \ref{S} preserves the distances, so
\[
d_M(f_m(x) \otimes 1, f(x) \otimes 1) \to 0 \implies  d_{acs}(\{\sdiag{f_m}\}_m, \sdiag{f_{m(n)}}) \to 0
\]
\[
\implies \{\sdiag{f_m}\}_m\acs  \sdiag{f_{m(n)}}.
\]
Using Lemma \ref{lambdacomp}, we get $\sdiag{f_m}\sim_\lambda f_m(x)$, so all the hypothesis of Lemma \ref{commuta} are satisfied and we obtain
\[
\sdiag{f_{m(n)}}\sim_\lambda f(x).
\]
Both the sequences $\serie D$ and $\sdiag{f_{m(n)}}$ converge piecewise to $f(x)$. This means that both the sequences of piecewise linear functions $g_n(x)$ and $h_n(x)$ associated to $\serie D$ and $\sdiag{f_{m(n)}}$ converge in measure to $f(x)$. The difference of the two functions converges in measure to $0$, but $g_n(x)-h_n(x)$ is the piecewise linear function associated to $D_n-D_n(f_{m(n)})$ so we get that
\[
\serie D - \sdiag{f_{m(n)}} \tratti 0 
\]
and using Lemma \ref{zerotratti}, we get
\[ \serie D - \sdiag{f_{m(n)}} \sim_\sigma 0. \]
Eventually, using Lemma \ref{zerosum} we conclude  
\[
\serie D = \sdiag{f_{m(n)}} + (\serie D - \sdiag{f_{m(n)}}) \sim_\lambda f(x)
\]
\end{proof}

\begin{theorem}\label{tratti}
Given $\serie D$ a sequence of diagonal matrices, and $k:[0,1]\to\f C$ any measurable function, then the following are equivalent
\begin{itemize}
\item $D_n\tratti k(x)$,
\item $D_n\GLT k(x)\otimes 1$. 
\end{itemize}
\end{theorem}
\begin{proof}
Given $k(x)$ a measurable function, we can always find a sequence of continuous functions $ k_m(x) $ that converges to $k(x)$ in measure.
Using Lemma \ref{mn}, there exists a map $m(n)$ such that
\[
\{D_n(k_{m(n)})\}_n\GLT k(x)\otimes 1\qquad \sdiag{k_{m(n)}}\tratti k(x).
\]
Using the algebra properties of the GLT space, we get
\[
\serie D\GLT k(x)\otimes 1 \iff \{D_n(k_{m(n)}) - D_n\}_n\GLT 0 
\] 
and thanks to Lemma \ref{zerotratti} we have
\[
 \{D_n(k_{m(n)}) - D_n\}_n\GLT 0 \iff \{D_n(k_{m(n)}) - D_n\}_n\tratti 0.\]
The piecewise convergence is linear: if $\serie A\tratti a(x)$ and $\serie B\tratti b(x)$, it is easy to see that $\serie A+\serie B\tratti a(x)+b(x)$. We can thus conclude that
 \[ \{D_n(k_{m(n)}) - D_n\}_n\tratti 0 \iff \serie  D\tratti k(x).
\]
\end{proof}

\section{Eigenvalues Order}

A sequence $\serie A$ may have several different spectral or singular values symbols. When dealing with real valued symbols, though, there is a preferred one, called \textit{decreasing rearrangement}.

\subsection{Decreasing Rearrangement}
Any measurable function $f:D\to \f R$ has a decreasing rearrangement that is a decreasing function $g:[0,1]\to \f R$ with the same distribution. We can define it as
\[
g(y) := \inf \left\{  z: \frac{\mu\{x: f(x)>z \}}{\mu(D)} \le y   \right\} .
\]
It is easy to check that it is a decreasing function, and moreover the following property hold:
\[
\frac{\mu\{x: f(x)>z \}}{\mu(D)} = \mu\{y: g(y)>z \} \qquad \forall z\in \f R.
\]
This result can be found in any graduate book of analysis, like \cite{decr}. If $f(x)$ is a spectral symbol for a sequence $\serie A$, then we can prove that also $g(y)$ is a spectral symbol for the same sequence. 

\begin{lemma}\label{rearr}
Given $\serie A$ a matrix sequence, $f:D\to \f R$ a measurable function with $D\cu\f R^n$ a set with finite non-zero measure, and $g:[0,1]\to \f R$ is its deacrising rearrangement, then
\[
\serie A\sim_\lambda f(x) \implies \serie A \sim_\lambda g(x).
\]
\end{lemma}
\begin{proof}
Let $\chi_{(a,b]}$ be the indicator function of the real interval $(a,b]$. We know that
\begin{align*}
\frac{1}{\mu(D)}\int_D \chi_{(a,b]}(f(x)) dx = \,\,& \frac{\mu\{x: f(x)>a \}   - \mu\{x: f(x)>b \} }{\mu(D)}\\
= \,\, & \mu\{y: g(y)>a \}  -\mu\{y: g(y)>b \}\\
= \,\, &   \int_0^1 \chi_{(a,b]}(g(y)) dy.
\end{align*}
This implies that given any step function $G$, obtained as a  linear combinations of indicator functions of real intervals $(a,b]$, we have that
\[ 
\frac{1}{\mu(D)}\int_D G(f(x)) dx = \int_0^1 G(g(y)) dy.
 \] 
Any real valued, compact supported and continuous function can be approximated in infinity norm arbitrarily well by step functions, so given $F\in C_c(\f R)$ and $\ve >0$, we can take a step function $G$ such that 
	\[
	\|F(x)-G(x)\|_\infty \le \ve.
	\] 
This means that
\begin{align*}
 \Bigg|  \frac{1}{\mu(D)}\int_D F(f(x)) & dx - \int_0^1 F(g(y)) dy  \Bigg| \le\\
& \left|  \frac{1}{\mu(D)}\int_D F(f(x)) dx - \frac{1}{\mu(D)}\int_D G(f(x)) dx  \right| \\
&+ \left|   \frac{1}{\mu(D)}\int_D G(f(x)) dx  - \int_0^1 G(g(y)) dy  \right| \\
&+ \left| \int_0^1 G(g(y)) dy - \int_0^1 F(g(y)) dy  \right|\le \ve + 0 +\ve
\end{align*}
for every $\ve >0$, so
\[
\frac{1}{\mu(D)}\int_D F(f(x)) dx = \int_0^1 F(g(y)) dy.
\] 
This concludes that for every $F\in C_c(\f C)$, 
\[
\lim_{n\to\infty} \frac{1}{n} \sum_{i=1}^{n} F(\lambda_i(A_n)) = \frac{1}{\mu(D)}\int_D F(f(x)) dx= \int_0^1 F(g(y)) dy.
\] 
\end{proof}

This means that we can work only with decreasing functions. In particular, for these functions it is possible to establish a connection between spectral symbols and piecewise convergence. The next lemma shows that the converse implication of Theorem \ref{trattilambda} holds with additional hypothesis.

\begin{lemma}\label{decr}
Let $\serie D$ be a sequence of diagonal hermitian matrices, and let $f_n$ be the maps associated to $d_n=diag(D_n)$. If $f_n$ and $f$ are decreasing real valued and measurable functions on $[0,1]$, then 
\[
\serie D \sim_\lambda f(x) \implies \serie D\tratti f(x)
\]
\end{lemma}
\begin{proof}
$f$ is a decreasing function, so the set of its discontinuity points have measure zero. If we suppose that $f_n$ do not converge to $f$ in measure, then in particular they do not converge punctually almost everywhere, so there is a set $S\cu [0,1]$ with non-zero measure such that
\[
\lim_{n\to\infty}f_n(x) \ne f(x) \qquad \forall x\in S. 
\]
Let $x_0\in S$ be a point where $f$ is continuous and $x_0\ne 0$. There exists $\ve>0$ such that $|f_n(x_0)-f(x_0)|>2\ve $ frequently. 
Surely, there exist infinite index $m$ such that $f_m(x_0)-f(x_0)>2\ve $ or there exist infinite index $m$ such that $f_m(x_0)-f(x_0)<-2\ve $. Suppose without loss of generality that the first case holds, and let $g_n$ be the subsequence of $g_n:=f_{k_n}$ such that
\[
g_n(x_0)-f(x_0)>2\ve\qquad  \forall n\in\f N.
\]
Using the continuity of $f$ in $x_0$, we can find $\delta>0$ such that
\[f(x_0-z)< f(x_0) + \ve \qquad \forall 0\le z\le 2\delta\]
and $x_0>4\delta$. Let $M:= f(x_0)+2\ve$, and divide in cases.\\

Case 1.) Suppose the sequence $g_n(\delta)$ has a limit point $z$.\\ Take $T > \max\{M+\ve, z\}$ and let $G\in C_c(\f R)$ such that $0\le G(x)\le 1$ for every $x\in\f R$, and
\[
G(x) = 
\begin{cases}
1 & x\in [M,T],\\
0 & x\not\in (M-\ve, T+\ve).
\end{cases}
\]
Notice that $g_n(x_0)>M$, $\delta < x_0$, and $g_n$ are decreasing, so $g_n(\delta)>M$, meaning that $z\ge M$. Moreover, $T>z$ implies that 
\[
g_n(x) \in [M,T] \qquad \forall x\in [\delta,x_0]\quad \text{frequently in }n.
\]  
If we consider only the points $i/n$ with $i=1,\dots,n$, then
\[
\frac 1 {k_n} \sum_{i=1}^{k_n} G\left(g_n\left(\frac i {k_n}  \right)  \right)
\ge
\frac{ 
\# \left\{ i: \frac{i}{k_n}\in [\delta,x_0]  \right\} 
}{k_n}
\ge  x_0 - \frac 3 2 \delta 
\] 
frequently when $ k_n>\frac{2}{\delta}$.
Moreover, we have
\[
 x > x_0-2\delta \implies f(x) \le f(x_0-2\delta) < f(x_0) + \ve = M-\ve  
\]
so
\[
\int_0^1 G(f(x)) dx \le x_0 - 2\delta.
\]
Finally, using the definition of spectral symbol, we obtain
\begin{align*}
x_0 - \frac 3 2 \delta &\le 
\lim_{n\to\infty} \frac{1}{n} \sum_{i=1}^{n}  G\left(f_n\left(\frac i {n}  \right)  \right)  \\
&= \lim_{n\to\infty} \frac{1}{n} \sum_{i=1}^{n} G(\lambda_i(D_n))\\
&= \int_0^1 G(f(x)) dx \\
&\le x_0 - 2\delta,
\end{align*}
that is an absurd.\\

Case 2.) Suppose the sequence $g_n(\delta)$ has not a limit point.\\
This means that any subsequence of $g_n(\delta)$ is not bounded, but $g_n(\delta)>M$ for every $n$, so they must diverge to $+\infty $. Let $\gamma >0$ such that $0<6\gamma < \delta$, and denote $R = f(\gamma)$, $T=f(1-\gamma)$. We can find a function $G\in C_c(\f R)$ 
such that $0\le G(x)\le 1$ for every $x\in\f R$, and
\[
G(x) = 
\begin{cases}
1 & x\in [T,R],\\
0 & x\not\in (T-\ve, R+\ve).
\end{cases}
\]
Since $g_n(\delta)$ diverges to infinity, then it will definitively be greater than $R+\ve$, and consequently
\[
x<\delta \implies g_n(x) \ge g_n(\delta) > R + \ve. 
\]
This means that
\[
\frac 1 {k_n} \sum_{i=1}^{k_n} G\left(g_n\left(\frac i {k_n}  \right)  \right)
\le
\frac{ 
\# \left\{ i: \frac{i}{k_n}\in [\delta,1]  \right\} 
}{k_n}
\le  1 - \frac 1 2 \delta 
\] 
frequently when $ k_n>\frac{2}{\delta}$. 
Moreover, we have
\[
 1-\gamma >x > \gamma  \implies T = f(1-\gamma )\le f(x) \le  f(\gamma) = R,  
\]
so
\[
\int_0^1 G(f(x)) dx \ge 1-2\gamma.
\]
Finally, using the definition of spectral symbol, we obtain
\begin{align*}
1-3\gamma > 1 - \frac 1 2 \delta &\ge 
\lim_{n\to\infty} \frac{1}{n} \sum_{i=1}^{n}  G\left(f_n\left(\frac i {n}  \right)  \right)  \\
&= \lim_{n\to\infty} \frac{1}{n} \sum_{i=1}^{n} G(\lambda_i(D_n))\\
&= \int_0^1 G(f(x)) dx \\
&\ge 1-2\gamma 
\end{align*}
that is an absurd.
\end{proof}

Using this result on decreasing functions, we can use the natural order of $\f R$ to obtain results regarding permutated diagonal matrices.

\subsection{GLT diagonal sequences}

Here we can use all the results proved in the other sections, to show that if a sequence of diagonal matrices has a spectral symbol, then we can reorder the elements on the diagonal to produce a GLT sequence with the same symbol.

\begin{theorem}
Given $\serie D$ a sequence of diagonal matrices with real entries such that $\serie{D}\sim_\lambda f(x)$, with $f:[0,1]\to\f R$,
then there exist permutation matrices $P_n$ such that
$$\{P_nD_nP_n^T\}\GLT f(x)\otimes 1.$$
\end{theorem}
\begin{proof}
Notice that the definition of spectral symbol depends only on the eigenvalues, and not on their order, so any permutation of the elements on the diagonal of $D_n$ does not change the spectral symbol.

Let $g(x)$ be the decreasing rearrangement of $f(x)$, and let $Q_n$ be permutation matrices such that $Q_nD_nQ_n^T$ have the eigenvalues sorted in decreasing order on the diagonal, for every $n$. Using Lemma \ref{rearr}, we know that
\[
\{Q_nD_nQ_n^T\}_n\sim_\lambda g(x)
\]
but now $g(x)$ and the functions associated to $Q_nD_nQ_n^T$ are all decreasing, so we can apply Lemma \ref{decr} and obtain
\[
\{Q_nD_nQ_n^T\}_n\tratti g(x).
\]
Thanks to Theorem \ref{tratti}, we also know that
\[
\{Q_nD_nQ_n^T\}_n\GLT g(x)\otimes 1.
\]
Since $f(x)$ is a measurable functions, then there exist a sequence $f_n(x)$ of continuous function that converge to $f(x)$ in measure, 
and thanks to Lemma \ref{mn}, we can find a sequence of diagonal matrices $\serie{D'}$ such that 
\[ 
\serie {D'}\GLT f(x)\otimes 1
 \]
and using the Theorems \ref{tratti} and \ref{trattilambda}, we get
\[
\serie {D'}\sim_\lambda f(x).
\]
We can then find again permutation matrices $S_n$ that sort the elements of $D'n$ in decreasing order, so that Lemma \ref{rearr}, Lemma \ref{decr} and Theorem \ref{tratti} lead to
\[
\{S_nD'_nS_n^T\}_n\GLT g(x)\otimes 1.
\]
Using the fact that the GLT space is an algebra, we obtain
\[
\{S_nD'_nS_n^T - Q_nD_nQ_n^T\}_n \GLT 0
\]
that, thanks to Lemma \ref{zerotratti}, leads to
\[
\{S_nD'_nS_n^T - Q_nD_nQ_n^T\}_n \sim_\lambda 0.
\]
Permuting again the entries, we have
\[
\{D'_n - S_n^TQ_nD_nQ_n^TS_n\}_n \sim_\lambda 0
\]
so we can use again  Lemma \ref{zerotratti}, that leads to
\[
\{D'_n - S_n^TQ_nD_nQ_n^TS_n\}_n \GLT 0
\]
but GLT is an algebra, so
\[
\{S_n^TQ_nD_nQ_n^TS_n\}_n \GLT f(x)\otimes 1.
\]
If $S_n^TQ_n=P_n$, we conclude that
\[
\{P_nD_nP_n^T\}_n \GLT f(x)\otimes 1.
\]
\end{proof}

\bibliography{diag.bib}

\bibliographystyle{plain}

\end{document}